\documentclass[11pt]{article}
\usepackage{amsthm, amsmath, amssymb, amsfonts, url, booktabs, tikz, setspace, fancyhdr, bm}
\usepackage{hyperref}
\usepackage{geometry}
\usepackage{enumerate}
\usepackage[shortlabels]{enumitem}
\usepackage[babel]{microtype}
\usepackage[english]{babel}
\usepackage{comment}
\usepackage{bbm}
\usepackage{csquotes}
\usepackage{mathabx}
\usetikzlibrary{calc, angles, quotes, intersections}
\usepackage{graphicx}
\usepackage{float}
\usepackage{xcolor}

\counterwithin{figure}{section}

\newtheorem{theorem}{Theorem}[section]

\newtheorem{lemma}[theorem]{Lemma}

\theoremstyle{definition}

\newtheorem*{defn-non}{Definition}

\usepackage[capitalise]{cleveref}

\newlist{Case}{enumerate}{2}
\setlist[Case, 1]{%
    label           =   {\bfseries Case \arabic*.},
    labelindent=1em ,labelwidth=1.3cm, labelsep*=1em, leftmargin =!
}
\setlist[Case, 2]{%
    label           =   {\bfseries Subcase \arabic{Casei}.\arabic*.},
    labelindent=-1em ,labelwidth=1.3cm, labelsep*=1em, leftmargin =!
}

\newcommand*{\abs}[1]{\lvert#1\rvert}

\newcommand{\R}{\mathbb{R}}

\newcommand{\aff}{\operatorname{aff}}

\newcommand{\E}{\mathbb{E}}
\newcommand{\crad}{\operatorname{crad}}
\newcommand{\dir}{\operatorname{dir}}
\newcommand{\canarrow}{\Longrightarrow}

\title{All simplices exhibit canonical Ramsey property}
\author{
Gennian Ge\thanks{School of Mathematical Sciences, Capital Normal University, Beijing 100048, China. Email: gnge@zju.edu.cn. Gennian Ge is supported by the National Key Research and Development Program of China under Grant 2025YFC3409900, the National Natural Science Foundation of China under Grant 12231014, and Beijing Scholars Program.}
\and 
Yang Shu\thanks{School of Mathematical Sciences, University of Science and Technology of China, Hefei, 230026, China. Email: shuyangyyyy@mail.ustc.edu.cn.}
\and
Zixiang Xu\thanks{School of Mathematical Sciences, Zhejiang University, Hangzhou, China. Email: zixiangxu@zju.edu.cn.}
\and
Wenjun Yu\thanks{Institute of Mathematics and Interdisciplinary Sciences, Xidian University, Xi’an, China. Email: yuwenjun@xidian.edu.cn.}
}
\date{}

\begin{document}

\maketitle

\begin{abstract}
We prove that all nondegenerate simplices have the canonical Ramsey property, thereby resolving a central open problem in canonical Euclidean Ramsey theory and providing a canonical counterpart to the celebrated simplex Ramsey theorem of Frankl and R\"{o}dl~[JAMS, 1990].
\end{abstract}

\section{Introduction}
Euclidean Ramsey theory was initiated in the
1970s by Erd\H{o}s, Graham, Montgomery, Rothschild, Spencer, and Straus through
their foundational trilogy~\cite{ErdosEtAl1975II,ErdosEtAl1975III,Erdos1973}.
It asks which finite configurations must occur monochromatically in every
finite coloring of a sufficiently high-dimensional Euclidean space. A finite
configuration is a finite subset of a Euclidean space, and a copy always means
a congruent copy. Let \(\E^n\) be the \(n\)-dimensional Euclidean space equipped with the Euclidean distance. A \emph{\(k\)-point nondegenerate simplex} is an affinely
independent set \(T=\{\boldsymbol{t}_{1},\ldots,\boldsymbol{t}_{k}\}\). A
configuration \(S\) is called \emph{Euclidean Ramsey} if,
for every positive integer \(q\), there is an integer \(n=n(S,q)\) such that
every \(q\)-coloring of \(\E^n\) contains a monochromatic copy of \(S\).
Unlike an abstract Ramsey problem, such a copy must realize all prescribed
pairwise distances simultaneously. This metric rigidity makes the
classification subtle: The authors in~\cite{ErdosEtAl1975II,ErdosEtAl1975III,Erdos1973} proved that every Euclidean Ramsey
configuration is spherical, meaning that it embeds on a sphere, and conjectured
that this necessary condition is also sufficient~\cite{Erdos1973}. Even
the basic cases of three-point patterns and equally spaced collinear sets have
generated substantial work, including planar two-color results for
triangles~\cite{JelinekEtAl2009,Shader1976}, asymmetric Ramsey problems for
lines~\cite{ConlonFox2019,ConlonFuhrer2024,ConlonWu2023}, and recent advances
on Euclidean Ramsey problems for arithmetic progressions
\cite{CurrierMooreYip2024,CurrierMooreYip2026,FuhrerToth2025}. In the
high-dimensional direction related to this paper, Frankl and R\"odl proved
that every nondegenerate simplex is exponentially Ramsey
\cite{frankl1990} and later established stronger Ramsey properties for
simplices~\cite{FranklRodl2004}. This makes simplices a principal positive
class in ordinary Euclidean Ramsey theory and provides the starting point for
the canonical problem considered here.

Euclidean Gallai--Ramsey theory was introduced by Mao, Ozeki, and Wang~\cite{mao2022}. For finite configurations \(K_{1},K_{2}\) and a positive integer \(r\), we write
\[
\E^{n}\xrightarrow{r}(K_{1};K_{2})_{\mathrm{GR}}
\]
if every coloring \(\chi:\E^{n}\to[r]\) contains either a monochromatic copy of \(K_{1}\) or a rainbow copy of \(K_{2}\), where a copy is \emph{rainbow} if its points receive pairwise distinct colors. In~\cite{geher2025} Geh\'{e}r, Sagdeev and T\'{o}th formally gave the following definition: a finite configuration \(S\) has the \emph{canonical Ramsey property}, or is \emph{canonically Ramsey} if there is an integer \(n_{0}=n_{0}(S)\) such that \(\E^{n}\xrightarrow{r}(S;S)_{\mathrm{GR}}\) for every positive integer \(r\) and every \(n\ge n_{0}\).

The essential requirement is that \(n_{0}\) be independent of the number of colors. Thus, unlike in the ordinary Euclidean Ramsey property, the ambient dimension must remain fixed while the number of colors is allowed to grow arbitrarily. For finite configurations \(X,A,T\), write \(X\canarrow(A;T)\) when every coloring of \(X\) with colors from an arbitrary set contains a monochromatic copy of \(A\) or a rainbow copy of \(T\). For finite configurations \(X,S\), write \(X\xrightarrow{q}S\) when every \(q\)-coloring of \(X\) contains a monochromatic copy of \(S\). Since every finite Euclidean configuration embeds in all sufficiently high dimensions, a finite witness \(W\canarrow(S;S)\) implies the canonical Ramsey property.

Cheng and Xu~\cite{chengxu2025} established the first dimension-independent Euclidean Gallai--Ramsey results for various configurations, including squares, right and acute triangles, and several classes of simplices, such as tetrahedra whose four heights all exceed the circumradius of the tetrahedron. Here the height from a vertex of a tetrahedron is its distance from the affine hull of the opposite face. Geh\'er, Sagdeev, and T\'oth~\cite{geher2025} formalized the canonical Ramsey property, proved it for every acute triangle already in \(\E^{3}\), and proved it for all hypercubes. Their results also cover rectangles whose squared aspect ratio is rational. Fang, Ge, Shu, Xu, Xu, and Yang~\cite{fang2025} subsequently proved that every triangle is canonically Ramsey already in \(\E^{4}\) and that every rectangle is canonically Ramsey. They also established the property for every tetrahedron whose largest height exceeds the smallest circumradius among its four triangular faces. Shaw~\cite{shaw2026} later proved that every cuboid is canonically Ramsey.

Among these results, the simplex problem is the central test case for the theory. Simplices are the basic affinely independent Euclidean configurations, and their edge lengths encode the complete metric data of finite point sets. More importantly, Frankl and R\"odl settled the ordinary Euclidean Ramsey problem for every nondegenerate simplex~\cite{frankl1990}, while the canonical theory had previously covered only all triangles and certain restricted families of higher-dimensional simplices. In light of the broader conjecture that every Euclidean Ramsey configuration is canonically Ramsey~\cite{fang2025}, simplices therefore form the first fundamental class for which a complete canonical theorem should be sought. The difficulty is precisely the new uniformity demanded by the canonical setting: one must control the full metric structure of a simplex in a fixed dimension independently of the number of colors.

We resolve this core problem by proving that every nondegenerate simplex has the canonical Ramsey property. 

\begin{theorem}\label{thm:main}
Let \(T\) be a finite nondegenerate simplex. 
There exists a finite configuration \(W=W(T)\) such that
\(
W\canarrow(T;T).
\) Consequently, there is an integer \(n_{0}=n_{0}(T)\) such that \(\E^{n}\xrightarrow{r}(T;T)_{\mathrm{GR}}\) for every positive integer \(r\) and every \(n\ge n_{0}\).
\end{theorem}
The quantitative results on \(n_{0}(T)\) produced by the proof are large, obtaining quantitative canonical Ramsey bounds for simplices remains an interesting problem. For example, when \(|T|=3\), it was shown in~\cite{fang2025} that \(n_{0}(T)\) can be \(4\).

\section{Proof of the main result}

We first list some geometric notation used in the proof. For every nonempty finite set \(X=\{\boldsymbol{x}_{1},\ldots,\boldsymbol{x}_{s}\}\), its affine hull is
\[
\aff(X)=\left\{\sum_{i=1}^{s}\alpha_{i}\boldsymbol{x}_{i}:\alpha_{1},\ldots,\alpha_{s}\in\R,\sum_{i=1}^{s}\alpha_{i}=1\right\}.
\]
If \(B=\{\boldsymbol{b}_{1},\ldots,\boldsymbol{b}_{N}\}\) is a nondegenerate simplex, then its circumcenter is the unique point \(\boldsymbol{o}_{B}\in\aff(B)\) equidistant from all vertices of \(B\). We denote this common distance, the circumradius of \(B\), by \(\crad(B)=\rho_{B}\). Thus the circumsphere of \(B\) is \(\{\boldsymbol{x}\in\aff(B):\lVert\boldsymbol{x}-\boldsymbol{o}_{B}\rVert=\rho_{B}\}\). Notice that \(\rho_{B}\) need not equal the radius of the smallest closed ball containing \(B\). We use \(\operatorname{dist}(\boldsymbol{x},L)=\inf\{\lVert\boldsymbol{x}-\boldsymbol{y}\rVert:\boldsymbol{y}\in L\}\) for the distance from a point to a nonempty set. For an affine subspace \(L\), its direction space is \(\dir(L)=\{\boldsymbol{x}-\boldsymbol{y}:\boldsymbol{x},\boldsymbol{y}\in L\}\). For \(\alpha>0\), write \(\alpha X=\{\alpha\boldsymbol{x}:\boldsymbol{x}\in X\}\). Products such as \(X\times Y\) are always orthogonal Cartesian products, and \(A^{\times s}\) denotes the product of \(s\) copies of \(A\). We next list several useful tools.

\subsection{Some useful tools}

We shall use the following finite form of the simplex Ramsey theorem of Frankl and R\"odl~\cite{frankl1990}.

\begin{theorem}[\cite{frankl1990}]\label{thm:frankl-rodl}
For every nondegenerate simplex \(S\) and every positive integer \(q\), there is a finite configuration \(X\) such that \(X\xrightarrow{q}S\). In particular, every nondegenerate simplex is Euclidean Ramsey.
\end{theorem}

We remark that Lemma~\ref{thm:frankl-rodl} follows directly from the super-Ramsey form proved by Frankl and R\"odl. For every nondegenerate simplex \(S\), there is a constant \(\delta_{S}>0\) and, for every sufficiently large \(n\), a finite configuration \(X_{n}\subseteq\E^{n}\) such that every subset \(Y\subseteq X_{n}\) containing no copy of \(S\) satisfies
\[
\abs{Y}<\abs{X_{n}}(1+\delta_{S})^{-n}.
\]
Given \(q\), choose a sufficiently large \(n\) such that \(q(1+\delta_{S})^{-n}<1\). If every color class of a \(q\)-coloring of \(X_{n}\) were \(S\)-free, their total cardinality would be strictly smaller than \(\abs{X_{n}}\), a contradiction. Hence \(X_{n}\xrightarrow{q}S\).

We also use the following standard contraction from the work of Frankl and R\"odl~\cite{frankl1990}. We include a short Gram-matrix proof to make the geometric dependence transparent.

\begin{lemma}[\cite{frankl1990}]\label{lem:contraction}
Let \(A=\{\boldsymbol{a}_{1},\ldots,\boldsymbol{a}_{k}\}\) be a nondegenerate simplex with \(k\ge2\). For every sufficiently small \(\lambda>0\), there is a nondegenerate simplex \(A^{-}=\{\boldsymbol{a}^{-}_{1},\ldots,\boldsymbol{a}^{-}_{k}\}\) such that
\[
\lVert\boldsymbol{a}^{-}_{i}-\boldsymbol{a}^{-}_{j}\rVert^{2}=\lVert\boldsymbol{a}_{i}-\boldsymbol{a}_{j}\rVert^{2}-\lambda^{2}
\]
for all \(1\le i<j\le k\).
\end{lemma}

\begin{proof}[Proof of Lemma~\ref{lem:contraction}]
Translate \(A\) so that \(\boldsymbol{a}_{1}=\boldsymbol{0}\), and let \(G\) be the positive definite Gram matrix of \(\boldsymbol{a}_{2},\ldots,\boldsymbol{a}_{k}\). Let \(I\) and \(J\) be, respectively, the identity and all-ones matrices of order \(k-1\). For sufficiently small \(\lambda>0\), the matrix
\[
G_{\lambda}=G-\frac{\lambda^{2}}{2}(I+J)
\]
remains positive definite. Choose linearly independent vectors \(\boldsymbol{a}^{-}_{2},\ldots,\boldsymbol{a}^{-}_{k}\) with Gram matrix \(G_{\lambda}\), and put \(\boldsymbol{a}^{-}_{1}=\boldsymbol{0}\).

The diagonal entries of \(G_{\lambda}\) are those of \(G\) minus \(\lambda^{2}\), while its off-diagonal entries are those of \(G\) minus \(\frac{\lambda^{2}}{2}\). Thus the required equality holds for every pair involving \(\boldsymbol{a}^{-}_{1}\). For \(2\le i<j\le k\), the polarization identity gives
\[
(G_{\lambda})_{ii}+(G_{\lambda})_{jj}-2(G_{\lambda})_{ij}=G_{ii}+G_{jj}-2G_{ij}-\lambda^{2}.
\]
Hence the squared distance also decreases by exactly \(\lambda^{2}\) for these pairs. Positive definiteness of \(G_{\lambda}\) gives the nondegeneracy of \(A^{-}\).
\end{proof}

The next result about contraction is of particular importance, which allows an ordinary finite Ramsey witness to be lifted to a witness that is itself a simplex.

\begin{lemma}\label{lem:simplex-witness}
Let \(A\) be a nondegenerate simplex with at least two vertices, and let \(q\) be a positive integer. There is a nondegenerate simplex \(B\) such that \(B\xrightarrow{q}A\).
\end{lemma}

\begin{proof}[Proof of Lemma~\ref{lem:simplex-witness}]
Choose \(\lambda>0\) and \(A^{-}\) as in Lemma~\ref{lem:contraction}. By Theorem~\ref{thm:frankl-rodl}, there is a finite configuration \(Y=\{\boldsymbol{y}_{1},\ldots,\boldsymbol{y}_{N}\}\) satisfying \(Y\xrightarrow{q}A^{-}\).

In a Euclidean space orthogonal to the one containing \(Y\), take an \(N\)-point regular simplex \(Z=\{\boldsymbol{z}_{1},\ldots,\boldsymbol{z}_{N}\}\) of edge length \(\lambda\). Define \(\boldsymbol{b}_{i}=(\boldsymbol{y}_{i},\boldsymbol{z}_{i})\) and \(B=\{\boldsymbol{b}_{1},\ldots,\boldsymbol{b}_{N}\}\). The set \(B\) is affinely independent. Indeed, an affine relation \(\sum_{i=1}^{N}\alpha_{i}\boldsymbol{b}_{i}=\boldsymbol{0}\) with \(\sum_{i=1}^{N}\alpha_{i}=0\) projects to \(\sum_{i=1}^{N}\alpha_{i}\boldsymbol{z}_{i}=\boldsymbol{0}\), so every \(\alpha_{i}\) is zero.

Given a \(q\)-coloring of \(B\), assign to each \(\boldsymbol{y}_{i}\) the color of \(\boldsymbol{b}_{i}\). There are indices \(i_{1},\ldots,i_{k}\) for which \(\{\boldsymbol{y}_{i_{1}},\ldots,\boldsymbol{y}_{i_{k}}\}\) is a monochromatic copy of \(A^{-}\). Relabel them so that \(\boldsymbol{y}_{i_{u}}\) corresponds to \(\boldsymbol{a}^{-}_{u}\). For \(1\le u<v\le k\),
\[
\lVert\boldsymbol{b}_{i_{u}}-\boldsymbol{b}_{i_{v}}\rVert^{2}=\lVert\boldsymbol{y}_{i_{u}}-\boldsymbol{y}_{i_{v}}\rVert^{2}+\lVert\boldsymbol{z}_{i_{u}}-\boldsymbol{z}_{i_{v}}\rVert^{2}=\lVert\boldsymbol{a}_{u}-\boldsymbol{a}_{v}\rVert^{2}.
\]
The corresponding vertices of \(B\) therefore form a monochromatic copy of \(A\).
\end{proof}

\subsection{Proof of Theorem~\ref{thm:main}}
We now formally provide the proof in details.
For \(k=1\), take \(W=T\). Henceforth assume \(k\ge2\), fix the ordering \(T=(\boldsymbol{t}_{1},\ldots,\boldsymbol{t}_{k})\), and we define its \(j\)-th successive height as
\[
h_{j}=\operatorname{dist}\bigl(\boldsymbol{t}_{j},\aff\{\boldsymbol{t}_{1},\ldots,\boldsymbol{t}_{j-1}\}\bigr)
\]
for \(2\le j\le k\). We use \(h_{*}=\min_{2\le j\le k}h_{j}>0\).

We first construct a tree-like simplex that forces either the desired rainbow target or a prescribed monochromatic simplex.

\begin{lemma}\label{claim:tree}
Let \(A\) and \(B\) be nondegenerate simplices. If \(B\xrightarrow{k-1}A\) and \(\crad(B)<h_{*}\), then there is a nondegenerate simplex \(R\) satisfying \(R\canarrow(A;T)\).
\end{lemma}

\begin{proof}[Proof of Lemma~\ref{claim:tree}]
Write \(B=\{\boldsymbol{b}_{1},\ldots,\boldsymbol{b}_{N}\}\) and \(\rho=\crad(B)\). Translate \(B\) so that its circumcenter is the origin. Thus \(B\) lies in an \((N-1)\)-dimensional linear space and \(\lVert\boldsymbol{b}_{i}\rVert=\rho\) for every \(i\).

Let \(\mathcal{T}\) be the rooted tree with levels \(1,\ldots,k\) in which the root is the unique vertex on level \(1\) and every vertex below level \(k\) has exactly \(N\) children. For each non-leaf vertex \(v\), let \(\mathcal{C}(v)\) denote its set of children, ordered from \(1\) to \(N\). The tree has \(\sum_{j=0}^{k-1}N^{j}\) vertices and \(\sum_{j=0}^{k-2}N^{j}\) non-leaf vertices.

For every non-leaf vertex \(v\), choose an \((N-1)\)-dimensional
linear space \(U_{v}\) and a unit vector
\(\boldsymbol{e}_{v}\in U_{v}^{\perp}\), and set
\(
E_{v}=U_{v}\oplus\operatorname{span}\{\boldsymbol{e}_{v}\}.
\) Choose these spaces to be mutually orthogonal and work in their orthogonal direct sum. Its dimension is \(N\sum_{j=0}^{k-2}N^{j}\), so the entire construction takes place in a finite-dimensional Euclidean space. Assign the origin to the root, and process the non-leaf vertices level by level, in an arbitrary fixed order within each level. In particular, every vertex precedes all of its descendants.

Suppose that \(v\) lies on level \(j-1\), where \(2\le j\le k\), and that the points already assigned to the path from the root to \(v\) are \((\boldsymbol{x}_{1},\ldots,\boldsymbol{x}_{j-1})\).
Inductively, this ordered tuple is congruent to \((\boldsymbol{t}_{1},\ldots,\boldsymbol{t}_{j-1})\). Since both tuples are affinely independent and have the same pairwise distances, the correspondence \(\boldsymbol{t}_{\ell}\mapsto\boldsymbol{x}_{\ell}\) for \(1\le \ell<j\) extends uniquely to an affine isometry
\[
\varphi_{v}:\aff\{\boldsymbol{t}_{1},\ldots,\boldsymbol{t}_{j-1}\}\longrightarrow\aff\{\boldsymbol{x}_{1},\ldots,\boldsymbol{x}_{j-1}\}.
\]
Let \(\boldsymbol{p}_{j}\) be the orthogonal projection of \(\boldsymbol{t}_{j}\) onto \(\aff\{\boldsymbol{t}_{1},\ldots,\boldsymbol{t}_{j-1}\}\), and put \(\boldsymbol{p}_{v}=\varphi_{v}(\boldsymbol{p}_{j})\).

Let \(S_{\mathrm{old}}\) be the set of points assigned before the children of \(v\) are added, and let \(V_{\mathrm{old}}=\dir(\aff(S_{\mathrm{old}}))\). The path from the root to \(v\) is contained in \(S_{\mathrm{old}}\), so \(\boldsymbol{p}_{v}\in\aff(S_{\mathrm{old}})\). An induction in the processing order shows that every point in \(S_{\mathrm{old}}\) belongs to the sum of the spaces \(E_{u}\) associated with vertices \(u\) processed before \(v\): this is clear for the root, and each newly assigned point is the sum of a point in the affine hull of its ancestors and a vector in the new space \(E_{u}\). Consequently, \(E_{v}\subseteq V_{\mathrm{old}}^{\perp}\).

Place in \(U_{v}\) a centered copy \(B_{v}=\{\boldsymbol{b}_{v,1},\ldots,\boldsymbol{b}_{v,N}\}\) of \(B\), labeled according to the fixed labeling of \(B\). Since \(\rho<h_{*}\le h_{j}\), the number \(c_{j}=(h_{j}^{2}-\rho^{2})^{1/2}\) is positive. For \(1\le i\le N\), put \(\boldsymbol{w}_{v,i}=\boldsymbol{b}_{v,i}+c_{j}\boldsymbol{e}_{v}\) and assign to the \(i\)-th child of \(v\) the point \(\boldsymbol{y}_{v,i}=\boldsymbol{p}_{v}+\boldsymbol{w}_{v,i}\).

We verify the required distances. For \(1\le \ell<j\), both \(\boldsymbol{p}_{v}\) and \(\boldsymbol{x}_{\ell}\) belong to \(\aff(S_{\mathrm{old}})\), whereas \(\boldsymbol{w}_{v,i}\in E_{v}\subseteq V_{\mathrm{old}}^{\perp}\). Therefore \(\boldsymbol{p}_{v}-\boldsymbol{x}_{\ell}\) and \(\boldsymbol{w}_{v,i}\) are orthogonal. Also, \(\lVert\boldsymbol{w}_{v,i}\rVert^{2}=\rho^{2}+c_{j}^{2}=h_{j}^{2}\). Since \(\varphi_{v}\) is an isometry and \(\boldsymbol{p}_{j}\) is the orthogonal projection of \(\boldsymbol{t}_{j}\), we obtain
\[
\lVert\boldsymbol{y}_{v,i}-\boldsymbol{x}_{\ell}\rVert^{2}=\lVert\boldsymbol{p}_{v}-\boldsymbol{x}_{\ell}\rVert^{2}+h_{j}^{2}=\lVert\boldsymbol{p}_{j}-\boldsymbol{t}_{\ell}\rVert^{2}+h_{j}^{2}=\lVert\boldsymbol{t}_{j}-\boldsymbol{t}_{\ell}\rVert^{2}.
\]
Together with the induction hypothesis for the distances among \(\boldsymbol{x}_{1},\ldots,\boldsymbol{x}_{j-1}\), this shows that \((\boldsymbol{x}_{1},\ldots,\boldsymbol{x}_{j-1},\boldsymbol{y}_{v,i})\) is congruent to \((\boldsymbol{t}_{1},\ldots,\boldsymbol{t}_{j})\). Furthermore, for \(1\le i<i'\le N\), we have \(\lVert\boldsymbol{y}_{v,i}-\boldsymbol{y}_{v,i'}\rVert=\lVert\boldsymbol{b}_{v,i}-\boldsymbol{b}_{v,i'}\rVert\), so the points assigned to \(\mathcal{C}(v)\) form a copy of \(B\). Processing all non-leaf vertices completes the construction and ensures that every path from the root to level \(j\) is congruent to \((\boldsymbol{t}_{1},\ldots,\boldsymbol{t}_{j})\).

We prove, in the same processing order, that all assigned points are affinely independent. The key observation is that the common nonzero \(\boldsymbol{e}_{v}\)-component turns the affine independence of \(B_{v}\) into the linear independence of \(\boldsymbol{w}_{v,1},\ldots,\boldsymbol{w}_{v,N}\). Indeed, if \(\sum_{i=1}^{N}\alpha_{i}\boldsymbol{w}_{v,i}=\boldsymbol{0}\), then projection onto \(\R\boldsymbol{e}_{v}\) gives \(c_{j}\sum_{i=1}^{N}\alpha_{i}=0\), and hence \(\sum_{i=1}^{N}\alpha_{i}=0\). Projection onto \(U_{v}\) now gives \(\sum_{i=1}^{N}\alpha_{i}\boldsymbol{b}_{v,i}=\boldsymbol{0}\). Since \(B_{v}\) is affinely independent, every \(\alpha_{i}\) is zero.

Assume inductively that \(S_{\mathrm{old}}\) is affinely independent when the children of \(v\) are added. Consider an affine relation
\[
\sum_{\boldsymbol{x}\in S_{\mathrm{old}}}\beta_{\boldsymbol{x}}\boldsymbol{x}+\sum_{i=1}^{N}\alpha_{i}(\boldsymbol{p}_{v}+\boldsymbol{w}_{v,i})=\boldsymbol{0}
\]
whose coefficients satisfy \(\sum_{\boldsymbol{x}\in S_{\mathrm{old}}}\beta_{\boldsymbol{x}}+\sum_{i=1}^{N}\alpha_{i}=0\). Because the total sum of the coefficients is zero, the relation may be rewritten as
\[
\sum_{\boldsymbol{x}\in S_{\mathrm{old}}}\beta_{\boldsymbol{x}}(\boldsymbol{x}-\boldsymbol{p}_{v})+\sum_{i=1}^{N}\alpha_{i}\boldsymbol{w}_{v,i}=\boldsymbol{0}.
\]
Here the first sum belongs to \(V_{\mathrm{old}}\), because \(\boldsymbol{p}_{v},\boldsymbol{x}\in\aff(S_{\mathrm{old}})\), while the second belongs to \(E_{v}\subseteq V_{\mathrm{old}}^{\perp}\). Both sums must therefore vanish. The linear independence of \(\boldsymbol{w}_{v,1},\ldots,\boldsymbol{w}_{v,N}\) gives \(\alpha_{i}=0\) for every \(i\). The remaining relation is an affine relation on \(S_{\mathrm{old}}\), so every \(\beta_{\boldsymbol{x}}\) is also zero. Induction proves that the configuration \(R\) consisting of all points assigned to \(\mathcal{T}\) is a nondegenerate simplex.

Color \(R\) with colors from an arbitrary set. If the points assigned to \(\mathcal{C}(v)\) for some non-leaf vertex \(v\) contain a monochromatic copy of \(A\), then we are done. Otherwise, we construct a path from the root to level \(k\) whose assigned points have pairwise distinct colors. Start with the root, for which this property is immediate. Suppose that a path \((v_{1},\ldots,v_{j-1})\) from the root to level \(j-1\) has been chosen with this property. If every point assigned to \(\mathcal{C}(v_{j-1})\) has one of the \(j-1\) colors already used on the path, then this copy of \(B\) uses at most \(j-1\le k-1\) colors. After relabeling these colors and allowing unused colors, this is a \((k-1)\)-coloring of \(B\). The relation \(B\xrightarrow{k-1}A\) then gives a monochromatic copy of \(A\), a contradiction. Hence some \(v_{j}\in\mathcal{C}(v_{j-1})\) is assigned a color not yet used on the path. Choosing \(v_{j}\) maintains the induction invariant. After \(k-1\) extensions, the assigned points on the resulting path form a rainbow copy of \(T\). Therefore \(R\canarrow(A;T)\).
\end{proof}

By Lemma~\ref{lem:simplex-witness} with \(q=k-1\), there is a simplex \(B_{0}\) such that \(B_{0}\xrightarrow{k-1}T\). Put \(\rho_{0}=\crad(B_{0})\), and choose a positive integer \(m\) such that \(\frac{\rho_{0}}{\sqrt{m}}<h_{*}\). Set \(A=\frac{1}{\sqrt{m}}T\) and \(B=\frac{1}{\sqrt{m}}B_{0}\). Then \(B\xrightarrow{k-1}A\) and \(\crad(B)=\frac{\rho_{0}}{\sqrt{m}}<h_{*}\). By Lemma~\ref{claim:tree}, there is a nondegenerate simplex \(R\) such that \(R\canarrow(A;T)\).

We next amplify the monochromatic alternative by synchronizing the positions of monochromatic product copies.

\begin{lemma}\label{claim:product}
Let \(A\) be nonempty and let \(\abs{T}\ge2\). If a nondegenerate simplex \(R\) satisfies \(R\canarrow(A;T)\), then for every positive integer \(s\) there is a finite configuration \(X_{s}\) such that \(X_{s}\canarrow(A^{\times s};T)\).
\end{lemma}

\begin{proof}[Proof of Lemma~\ref{claim:product}]
We induct on \(s\). Take \(X_{1}=R\). Suppose that \(X_{s}\) has been constructed. A \emph{position} is a subset of \(X_{s}\) that is congruent to \(A^{\times s}\). Let \(\mathcal{P}_{s}\) be the finite family of all positions and put \(M_{s}=\abs{\mathcal{P}_{s}}\). We have \(M_{s}\ge1\): the constant coloring of \(X_{s}\) has no rainbow \(T\), so the induction hypothesis gives a monochromatic copy of \(A^{\times s}\).

Since \(R\) is a nondegenerate simplex, Theorem~\ref{thm:frankl-rodl} gives a finite configuration \(Q_{s}\) satisfying \(Q_{s}\xrightarrow{M_{s}}R\). Set \(X_{s+1}=Q_{s}\times X_{s}\), where the two factors lie in orthogonal spaces.

Color \(X_{s+1}\) arbitrarily and assume that it has no rainbow copy of \(T\). For every \(\boldsymbol{u}\in Q_{s}\), identify the fiber \(\{\boldsymbol{u}\}\times X_{s}\) with \(X_{s}\). The fiber has no rainbow \(T\), so choose \(F_{\boldsymbol{u}}\in\mathcal{P}_{s}\) such that \(\{\boldsymbol{u}\}\times F_{\boldsymbol{u}}\) is monochromatic, and color \(\boldsymbol{u}\) by the position \(F_{\boldsymbol{u}}\). The relation \(Q_{s}\xrightarrow{M_{s}}R\) gives a copy \(R'\subseteq Q_{s}\) of \(R\) on which this position-coloring is constant. Hence there is a single position \(F\in\mathcal{P}_{s}\) such that \(\{\boldsymbol{u}\}\times F\) is monochromatic for every \(\boldsymbol{u}\in R'\).

Let \(c(\boldsymbol{u})\) denote the color of the fiber \(\{\boldsymbol{u}\}\times F\). This is a coloring of \(R'\). It has no rainbow \(T\). Indeed, if \(\boldsymbol{u}_{1},\ldots,\boldsymbol{u}_{\abs{T}}\) formed such a copy, then for any fixed \(\boldsymbol{x}\in F\), the points \((\boldsymbol{u}_{1},\boldsymbol{x}),\ldots,(\boldsymbol{u}_{\abs{T}},\boldsymbol{x})\) would form a rainbow copy of \(T\) in \(X_{s+1}\). Identifying \(R'\) with \(R\) by a congruence, the relation \(R\canarrow(A,T)\) shows that \(c\) contains a monochromatic copy \(A'\subseteq R'\) of \(A\). Every point of \(A'\times F\) has the same color. The product of congruences from \(A\) to \(A'\) and from \(A^{\times s}\) to \(F\) is an isometry on the orthogonal product, so \(A'\times F\) is congruent to \(A^{\times(s+1)}\). This proves the induction step.
\end{proof}

Apply Lemma~\ref{claim:product} with \(s=m\). We obtain a finite configuration \(W=X_{m}\) such that \(W\canarrow(A^{\times m};T)\). The product \(A^{\times m}\) contains a diagonal copy of \(T\). Indeed, write \(\boldsymbol{a}_{i}=\frac{1}{\sqrt{m}}\boldsymbol{t}_{i}\), and let \(\boldsymbol{d}_{i}=(\boldsymbol{a}_{i},\ldots,\boldsymbol{a}_{i})\in A^{\times m}\) have \(m\) identical coordinates. For \(1\le i<j\le k\),
\[
\lVert\boldsymbol{d}_{i}-\boldsymbol{d}_{j}\rVert^{2}=m\lVert\boldsymbol{a}_{i}-\boldsymbol{a}_{j}\rVert^{2}=\lVert\boldsymbol{t}_{i}-\boldsymbol{t}_{j}\rVert^{2}.
\]
Thus \(D=\{\boldsymbol{d}_{1},\ldots,\boldsymbol{d}_{k}\}\) is congruent to \(T\).

Consider any coloring of \(W\). If it contains a rainbow \(T\), we are done. Otherwise, it contains a monochromatic copy of \(A^{\times m}\). Under a congruence from \(A^{\times m}\) to this copy, the image of \(D\) is a monochromatic copy of \(T\). Hence \(W\canarrow(T;T)\).

Finally, choose \(n_{0}\) so that \(W\) embeds in \(\E^{n_{0}}\). For \(n\ge n_{0}\), identify \(\E^{n_{0}}\) with the coordinate subspace \(\{(\boldsymbol{x},\boldsymbol{0}):\boldsymbol{x}\in\E^{n_{0}}\}\subseteq\E^{n}\). For every positive integer \(r\), restricting an arbitrary \(r\)-coloring of \(\E^{n}\) to an embedded copy of \(W\) gives a monochromatic or rainbow copy of \(T\). This proves the theorem.

\section{Concluding remarks}
The proof of Theorem~\ref{thm:main} yields more than its diagonal statement. We conclude by describing several consequences of the same ideas. 

\begin{itemize}
\item \textbf{Asymmetric form.} For any two nondegenerate simplices \(S\) and \(T\), there is a finite configuration \(W=W(S,T)\) such that \(W\canarrow(S;T)\). Indeed, after translation, place a copy \(S'\) of \(S\) in a linear space \(V_{S}\) and a copy \(T'\) of \(T\) in \(V_{T}+\boldsymbol{e}\), where \(V_{S}\), \(V_{T}\), and the direction spanned by \(\boldsymbol{e}\) are mutually orthogonal. Projection onto these three directions shows that \(U=S'\cup T'\) is affinely independent. Thus \(U\) is a nondegenerate simplex containing copies of both \(S\) and \(T\). Applying Theorem~\ref{thm:main} to \(U\), every monochromatic copy of \(U\) contains a monochromatic copy of \(S\), while every rainbow copy of \(U\) contains a rainbow copy of \(T\).

\item \textbf{Simplex witnesses.} The witness above can itself be chosen to be a nondegenerate simplex: for any two nondegenerate simplices \(S\) and \(T\), there is a nondegenerate simplex \(R=R(S,T)\) such that \(R\canarrow(S;T)\). To see this, choose a common sufficiently small \(\lambda>0\) and use Lemma~\ref{lem:contraction} to obtain \(S^{-}\) and \(T^{-}\) by decreasing every squared edge length by \(\lambda^{2}\). The asymmetric form gives a finite configuration \(Y=\{\boldsymbol{y}_{1},\ldots,\boldsymbol{y}_{N}\}\) satisfying \(Y\canarrow(S^{-};T^{-})\). In a space orthogonal to the one containing \(Y\), take an \(N\)-point regular simplex \(Z=\{\boldsymbol{z}_{1},\ldots,\boldsymbol{z}_{N}\}\) of edge length \(\lambda\). Then \(R=\{(\boldsymbol{y}_{i},\boldsymbol{z}_{i}):1\le i\le N\}\) is affinely independent, as in Lemma~\ref{lem:simplex-witness}. The second coordinates add exactly \(\lambda^{2}\) to every squared distance, so monochromatic copies of \(S^{-}\) and rainbow copies of \(T^{-}\) lift to monochromatic copies of \(S\) and rainbow copies of \(T\), respectively.

\end{itemize}

We suspect that every Euclidean Ramsey configuration is also canonically Ramsey, also see~\cite[Conjecture 1]{fang2025}. Theorem~\ref{thm:main} answers this question for every nondegenerate simplex. Notice that any Euclidean Ramsey configuration is necessarily spherical, but the proof in this paper uses affine independence at several essential points, including the contraction, the construction of simplex witnesses, and the tree-like embedding. These steps do not extend directly to affinely dependent spherical configurations. It would also be interesting to obtain useful bounds for the minimum size or ambient dimension of a canonical witness in terms of the number of vertices and geometric parameters of the target configuration.

\bibliographystyle{abbrv}
\bibliography{bib}

\end{document}